\documentclass[a4, english]{amsart}

\usepackage{textcomp}

\usepackage[utf8]{inputenc}
\usepackage[T1]{fontenc}

\usepackage{xfrac}
\usepackage{mathtools}
\usepackage{amssymb}
\usepackage{amsthm}
\usepackage{thmtools}

\usepackage{multicol}

\usepackage{tikz-cd}

\usepackage[style=alphabetic, url=false]{biblatex}

\usepackage{bbold}
\usepackage{enumitem}

\usepackage{hyperref}
\usepackage{cleveref}

\usepackage{babel}

\newcommand{\spacedequal}[0]{\quad=\quad}
\newcommand{\xspacedequal}[1]{\ensuremath{\quad\overset{\scriptscriptstyle #1}{=}\quad}}

\newcommand{\catname}[1]{\ensuremath{\normalfont\textbf{#1}}}

\newcommand{\catcat}[0]{\catname{Cat}}

\newcommand{\catop}[1]{{#1}^{\normalfont\text{op}}}

\newcommand{\functor}[1]{\ensuremath{\mathcal{#1}}}

\newcommand{\ffunctor}[1]{\ensuremath{\mathcal{#1}}}

\numberwithin{equation}{section}

\swapnumbers

\declaretheorem[style=plain,sibling=equation]{theorem}

\declaretheorem[style=plain,sibling=equation]{proposition}

\declaretheorem[style=definition,sibling=equation]{definition}

\declaretheorem[style=remark,sibling=equation]{remark}
\declaretheorem[style=definition, sibling=equation]{notation}

\DeclareMathOperator{\id}{Id}
\DeclareMathOperator*{\colim}{colim}
\DeclareMathOperator*{\llim}{bilim}
\DeclareMathOperator*{\ccolim}{bicolim}

\author{Jun Maillard}
\date{\today}
\title{On 2-final 2-functors}

\address{\ \medbreak
  Univ.\ Lille, CNRS, UMR 8524 - Laboratoire Paul Painlevé, F-59000 Lille, France}
\email{jun.maillard@gmail.com}
\urladdr{http://math.univ-lille1.fr/~{}jmaillard}
\thanks{Author supported by Project ANR ChroK (ANR-16-CE40-0003) and Labex CEMPI (ANR-11-LABX-0007-01).}


\usetikzlibrary{shapes.misc, positioning}

\tikzset{tip/.style={}}
\tikzset{dot/.style={circle, fill, minimum size=4pt, inner sep=0pt, outer sep=0pt}}
\tikzset{iden/.style={draw, circle, minimum size=4pt, inner sep=0pt, outer sep=0pt}}
\tikzset{naturaltr/.style={draw, rounded rectangle,
    minimum height=0.6cm, minimum width=1.8cm, inner sep=1mm}}

\allowdisplaybreaks

\begin{document}
\maketitle

A final functor between categories $\functor{F} \colon \catname{A} \to
\catname{B}$ is a functor that allows the restriction of diagrams on $\catname{B}$ to $\catname{A}$
without changing their colimits. More precisely, the functor $\functor{F}$ is final if, for
any diagram $\functor{D} \colon \catname{B} \to \catname{B}$, there is a
canonical isomorphism
\[
  \colim_{\catname{B}}\functor{D} \cong \colim_{\catname{A}}\functor{D} \circ \functor{F}
\]
where either colimit exists whenever the other one does. There is a classical
criterion for final
functors~\autocite[§IX.3]{maclanesaundersCategoriesWorkingMathematician1971}:
a functor $\functor{F} \colon \catname{A} \to \catname{B}$ is final if and only if, for any object $b \in
\catname{B}$, the slice category $b / \functor{F}$ is nonempty and connected.
Such a criterion also exists for $(\infty,
1)$-categories~\autocite[§4.1]{lurieHigherToposTheory2009}: an $(\infty,
1)$-functor $\ffunctor{F} \colon \catname{A} \to \catname{B}$ is final (with respect to any
$\infty$-diagram) if and only if for any object $b \in \catname{B}$, the slice
$\infty$-category $b / \functor{F}$ is weakly contractible. One would expect a
similar result for any dimension: an $(n, 1)$-functor $\ffunctor{F} \colon \catname{A} \to
\catname{B}$ is final (with respect to any $n$-diagram) if and only if, for any
object $b \in \catname{B}$, the slice $(n, 1)$-category $b / \functor{F}$ is
nonempty and has trivial homotopy groups $\pi_k$ for $0 \leq k \leq n-1$. Note
that these are not consequences of the known criterion for $1$-functors and
$(\infty, 1)$-functors (see \cref{rmkHigherPi} and \cref{rmk12Final}). This paper presents a
combinatorial proof in the case $n = 2$~(\cref{thmFinal}). An application of
this criterion will appear in my Ph.D. thesis~\autocite{maillard}.

\section{Bicategorical notions}

We will follow the naming conventions of \autocite{johnsonDimensionalCategories2021}
for bicategories. In particular, the terms $2$-category and $2$-functor
denote the \emph{strict} ones. We will use the term $(2,1)$-category for a
$2$-category with only invertible $2$-morphisms, and the term $(2,1)$-functor
for a $2$-functor between $(2,1)$-categories.

We recall some usual constructions and properties of $2$-categories we will use,
and introduce some notations.

\begin{notation}
  The symbol $\simeq$ denotes an isomorphism between two objects (in the $1$-categorical
  sense).

  The symbol $\cong$ denotes an equivalence between two objects (in the
  $2$-categorical sense).
\end{notation}

\begin{definition}\label{defOpCat}
 Let $\catname{C}$ be a 2-category. The opposite 2-category of $\catname{C}$, written
 $\catop{\catname{C}}$, is the 2-category with
 \begin{itemize}
 \item \emph{Objects:} the objects of $\catname{C}$
 \item \emph{Hom-categories:} $\catop{\catname{C}}(A, B) = \catname{C}(B, A)$
 \end{itemize}
\end{definition}

\begin{notation}
  We write $[\catname{A},\catname{B}]$ for the $2$-category of pseudofunctors,
  pseudonatural transformations and modifications, between two $2$-categories
  $\catname{A}$ and $\catname{B}$.
\end{notation}

\begin{notation}
 Let $\catname{I}, \catname{C}$ be $2$-categories and $T$ be an object of
 $\catname{C}$. We denote by $\Delta{T}$ the constant functor $\catname{I} \to
 \catname{C}$ with value $T$.
\end{notation}

\begin{definition}
 Let $\catname{I}, \catname{C}$ be $2$-categories and $\ffunctor{D} \colon
 \catname{I} \to \catname{C}$ be a $2$-functor. A \emph{(pseudo) bicolimit} of
 $\ffunctor{D}$ is an object $L$ of $\catname{C}$ and a family of equivalences
 \[
   \Psi_{T} \colon \catname{C}(L, T) \cong [\catname{I},\catname{C}](\ffunctor{D}, \Delta{T})
 \]
 pseudonatural in $T$. When it exists, the bicolimit of $\ffunctor{D}$ is unique
 up to equivalence and the object $L$ is noted $\ccolim_{\catname{I}}\ffunctor{D}$.
\end{definition}

\begin{notation}
  We will use the term \emph{$2$-diagram} to denote a $2$-functor we introduce with
  the intent to take its bicolimit.

  We will use the term \emph{cone under $\ffunctor{D}$ with vertex $T$} to denote
  objects of the category of pseudonatural transformations and modifications $[\catname{I},\catname{C}](\ffunctor{D}, \Delta{T})$.
\end{notation}

\begin{definition}
  Let $\catname{C}$ be a $(2,1)$-category. Fix a $(2,1)$-functor $\ffunctor{F}
  \colon \catname{I} \to \catname{C}$ and an object $c$ of $\catname{C}$. The
  \emph{slice} $c / \ffunctor{F}$ is the $(2, 1)$-category with:
  \begin{itemize}
  \item \emph{Objects:} the pairs $(i, f)$ consisting of an object $i$ of
    $\catname{I}$ and a morphism $f \colon c \to \ffunctor{F}i$
  \item \emph{Morphisms $(i, f) \to (i', f')$:} the pairs $(u, \mu)$ consisting
    of a morphism $u \colon i \to i'$ of $\catname{I}$ and a 2-isomorphism $\mu \colon f' \to
    \ffunctor{F}(u)f$ of $\catname{C}$:
    \[
      \begin{tikzcd}
        \ffunctor{F}i \arrow[rr, "\ffunctor{F}u"] & & \ffunctor{F}i' \\
        & c \arrow[lu, "f", bend left] \arrow[ru, "f'"{name=n}, bend right, swap]
        \arrow[from=n, to=1-1, Rightarrow, "\mu", shorten=0.5cm] &
      \end{tikzcd}
    \]
  \item \emph{2-Morphisms $(u, \mu) \Rightarrow (v, \nu)$:} the 2-morphisms
    $\alpha \colon u \Rightarrow v$ satisfying:
    \[
      \begin{tikzcd}[row sep=2cm]
        \ffunctor{F}i
          \arrow[rr, "\ffunctor{F}v"{name=nv}, bend left]
          \arrow[rr, "\ffunctor{F}u"{name=nu}, bend right, swap] & & \ffunctor{F}i' \\
        & c \arrow[lu, "f", bend left] \arrow[ru, "f'"{name=n}, bend right, swap]
        \arrow[from=n, to=1-1, Rightarrow, "\mu", shorten=0.5cm, shift left=0.4cm]
        \arrow[from=nu, to=nv, Rightarrow, "\ffunctor{F}\alpha", shorten=0.2cm]&
      \end{tikzcd} =
      \begin{tikzcd}[row sep=2cm]
        \ffunctor{F}i \arrow[rr, "\ffunctor{F}v", bend left] & & \ffunctor{F}i' \\
        & c \arrow[lu, "f", bend left] \arrow[ru, "f'"{name=n}, bend right, swap]
        \arrow[from=n, to=1-1, Rightarrow, "\nu", shorten=0.5cm] &
      \end{tikzcd}
    \]
  \item Compositions are induced by the compositions of $\catname{I}$ and $\catname{C}$.
  \end{itemize}
  A slice $2$-category $c/\ffunctor{F}$ is endowed with a canonical forgetful $2$-functor:
  \[
    \left\{
      \begin{array}{lll}
        c / \ffunctor{F} & \to & \catname{I} \\
        (i, f) & \mapsto & i \\
        (u, \mu) & \mapsto & u \\
        \alpha & \mapsto & \alpha
      \end{array}
    \right.
  \]
\end{definition}

\section{Combinatorial paths and homotopies}

A 2-category $\catname{C}$ has an associated CW-complex $|\catname{C}|$, defined using the
Duskin nerve~\autocite[§5.4]{johnsonDimensionalCategories2021}, which maps
objects of $C$ to vertices, $1$-morphisms to 1-simplices and $2$-morphisms to
2-simplices. There are thus notions of paths and homotopies of paths in
$\catname{C}$. We give in this section a combinatorial approach to these, for
$(2,1)$-categories.

We fix a $(2,1)$-category $\catname{C}$.

\begin{definition}
  A \emph{path} (of 1-morphism) in \catname{C} is a finite sequence of objects $(a_i)_{0 \leq i \leq
    n}$ and a family of pairs $(\varepsilon_i, f_i)_{1 \leq i \leq n}$ of a sign
  $\varepsilon \in \{ -1, 1 \}$ and a morphism
  \[
    f_i \colon \left\{
      \begin{array}{llll}
        a_{i-1} & \to & a_i & \text{if } \varepsilon_i = 1 \\
        a_{i} & \to & a_{i-1} & \text{if } \varepsilon_i = -1 \\
      \end{array}
    \right.
  \]
  Such a path is said to have source $a_0$ and target $a_n$.
\end{definition}

\begin{notation}
 We write $p \colon a_0 \leadsto a_n$ to denote a path with source $a_0$ and
 target $a_n$.
\end{notation}

A path can be pictured as a zig-zag of morphisms (potentially with consecutive
morphisms in the same direction):
\[
  a_0 \xrightarrow{f_1} a_1 \xleftarrow{f_2} a_2 \xleftarrow{f_3} \ldots
  \xrightarrow{f_n} a_n
\]
Following the usual conventions, left-to-right arrows represents pairs with
$\varepsilon = 1$ and right-to-left arrows pairs with $\varepsilon = -1$. The
empty path (at an object $a$) should be represented by $a$.

There is an obvious notion of concatenation of paths with compatible target and
source, given by the concatenation of the sequence of morphisms.

\begin{definition}
  We say two path $p, p'$ of $\catname{C}$ are \emph{elementary homotopic},
  written $p \sim_{\text{elem}} p'$, in any of the following cases:
  \begin{enumerate}[label=(\arabic*)]
  \item \label{ElemHtpyId1} $a \xrightarrow{\id} a \: \sim_{\text{elem}} \: a$, for any object $a$
  \item \label{ElemHtpyId2} $a \xleftarrow{\id} a \: \sim_{\text{elem}} \: a$, for any object $a$
  \item $a_0 \xrightarrow{f_1} a_1 \xrightarrow{f_2} a_2 \: \sim_{\text{elem}} \: a_0
    \xrightarrow{f_2f_1} a_2$, for any composable pair $f_1, f_2$ of morphisms
  \item $a_0 \xleftarrow{f_1} a_1 \xleftarrow{f_2} a_2 \: \sim_{\text{elem}} \: a_0
    \xleftarrow{f_1f_2} a_2$, for any composable pair $f_1, f_2$ of morphisms
  \item \label{ElemHtpyFace} $a_0 \xleftarrow{u} a_1 \xrightarrow{v} a_2 \:
    \sim_{\text{elem}} \:  a_0
    \xrightarrow{u'} a_1' \xleftarrow{v'} a_2$, for any 2-isomorphism
    \[
      \begin{tikzcd}
        & a_1' & \\
        a_0 \arrow[ru, "u'"] \arrow[rr, Rightarrow, shorten=0.5cm] & {} & a_2
        \arrow[lu, "v'", swap] \\
        & a_1 \arrow[lu, "u"] \arrow[ru, "v", swap] &
      \end{tikzcd}
    \]
  \end{enumerate}
  We then define a \emph{homotopy} relation $\sim$ on paths as the smallest congruent
  (for the concatenation of paths), reflexive, symmetric and transitive relation
  encompassing the relation $\sim_{\text{elem}}$.
\end{definition}

\begin{remark}
  We should pause to consider two consequences of \ref{ElemHtpyFace}:
  \begin{itemize}
  \item a 2-morphism $\begin{tikzcd}[cramped] a_0 \arrow[r, bend left=2cm,
      "f_0"{name=n0}] \arrow[r, bend right=2cm, swap, "f_1"{name=n1}] & a_1
      \arrow[from=n0, to=n1, Rightarrow, shorten=0.2cm] \end{tikzcd}$ can be
    arranged into the following square:
    
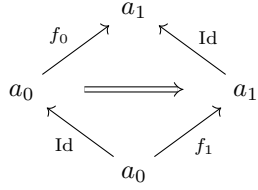
\[
      \begin{tikzcd}
        & a_1 & \\
        a_0 \arrow[ru, "f_0"] \arrow[rr, Rightarrow, shorten=0.5cm] & {} & a_1
        \arrow[lu, "\id", swap] \\
        & a_0 \arrow[lu, "\id"] \arrow[ru, "f_1", swap] &
      \end{tikzcd}
    \]
    This shows, together with \ref{ElemHtpyId1} and \ref{ElemHtpyId2}, that $a_0
    \xrightarrow{f_0} a_1 \: \sim \: a_0 \xrightarrow{f_1} a_1$, as one would expect.
  \item a 1-morphism $a_0 \xrightarrow{f} a_1$ can be used to form the square:
    
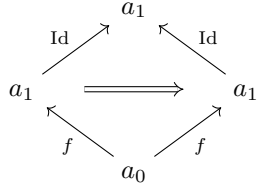
\[
      \begin{tikzcd}
        & a_1 & \\
        a_1 \arrow[ru, "\id"] \arrow[rr, Rightarrow, shorten=0.5cm] & {} & a_1
        \arrow[lu, "\id", swap] \\
        & a_0 \arrow[lu, "f"] \arrow[ru, "f", swap] &
      \end{tikzcd}
    \]
    Once again using \ref{ElemHtpyId1} and \ref{ElemHtpyId2}, this proves that
    $(a_1 \xleftarrow{f} a_0 \xrightarrow{f} a_1) \sim a_1$. A similar argument
    (putting $f$ on the upper side of the square) shows that $(a_0
    \xrightarrow{f} a_1 \xleftarrow{f} a_0) \sim a_0$. Hence, up to homotopy,
    the paths $a_0 \xrightarrow{f} a_1$ and $a_1 \xleftarrow{f} a_0$ are mutual
    inverses.
  \end{itemize}
\end{remark}

Note that for two paths to be homotopic, they must have the same source and the
same target.

It is natural to look for a category of paths up-to homotopy:
\begin{definition}
  The \emph{(algebraic) fundamental groupoid} $\Pi_1(\catname{C})$ of
  $\catname{C}$ is the 1-category with:
  \begin{itemize}
  \item \emph{Objects:} the objects of $\catname{C}$.
  \item \emph{Morphisms:} the classes of paths between objects modulo the
    homotopy relation.
  \item \emph{Composition} is induced by the concatenation of paths.
  \end{itemize}
\end{definition}
\begin{definition}\label{defConn}
  The (2,1)-category $\catname{C}$ is said to be \emph{connected} if for any
  pair of objects $a, a'$ there is a path with source $a$ and target $a'$.

  The (2,1)-category $\catname{C}$ is said to be \emph{simply connected} if $p
  \sim p'$ for any pair of
  paths $p, p'$ with same source and same target.
\end{definition}

\begin{remark}\label{rmkPiConn}
 A (2,1)-category $\catname{C}$ is nonempty, connected and simply connected if and
 only if its fundamental groupoid $\Pi_1(\catname{C})$ is equivalent to $1$, the
 category with exactly one object and one morphism.
\end{remark}

\begin{remark}\label{rmkHigherPi}
  Given a (2,1)-category $\catname{C}$ which is nonempty, connected and simply
  connected, its nerve $|\catname{C}|$ is not necessarily weakly contractible.
  Indeed higher homotopy groups may be nontrivial. For instance, one can realize
  the sphere $S^2$ as the nerve of the $(2,1)$-category with two objects, two
  parallel 1-morphisms between these objects, and two parallel 2-isomorphisms
  between these 1-morphisms.
\end{remark}

\begin{remark}
  For any algebraic path $p$ in $\catname{C}$, there is an
  associated topological path $|p| \colon I \to |\catname{C}|$. The following
  assertions, which should result from simplicial approximation, motivate
  the definitions of this section:

  The $2$-category $\catname{C}$ is connected (resp. simply connected) if and
  only if the CW-complex $|\catname{C}|$ is connected (resp. simply connected).

  Two algebraic paths $p, p'$ in $\catname{C}$ are homotopic if and only if the
  topological paths $|p|, |p'|$ are homotopic.

  The categories $\Pi_1(\catname{C})$ and $\Pi_1(|\catname{C}|)$ are equivalent.
\end{remark}

\section{A criterion for 2-final 2-functors}

\begin{definition}\label{def2Final}
  A $2$-functor $\ffunctor{F} \colon \catname{A} \to \catname{B}$ between
  $(2,1)$-categories is \emph{$2$-final} if for any 2-diagram $\ffunctor{D} \colon \catname{B} \to
  \catname{E}$, the pseudo bicolimits $\ccolim_{\catname{B}} \ffunctor{D}$ and
  $\ccolim_{\catname{A}}{\ffunctor{D} \circ \ffunctor{F}}$ each exists if and only if the other one
  exists, and the canonical comparison morphism
  \[
    \ccolim_{\catname{A}}{\ffunctor{D} \circ \ffunctor{F}} \to \ccolim_{\catname{B}}{\ffunctor{D}}
  \]
  is an equivalence.
\end{definition}

\begin{remark}
 In the above definition, $\catname{E}$ is only assumed to be a $2$-category.
 However, since $\catname{B}$ is a $(2,1)$-category, the \emph{pseudo}
 bicolimits can be equivalently computed in $\catname{E}_g$, the
 $(2,1)$-category with the objects of $\catname{E}$, the 1-morphisms
 of $\catname{E}$, and the \emph{invertible} 2-morphisms of $\catname{E}$. Hence we could assume
 $\catname{E}$ to be a $(2,1)$-category, without changing the meaning of the definition.
\end{remark}

\begin{remark}\label{rmk12Final}
  A $1$-final $1$-functor $\ffunctor{F} \colon \catname{A} \to \catname{B}$ between
  $1$-categories is a functor such that, for any diagram $\ffunctor{D} \colon \catname{B} \to
  \catname{E}$, the colimits $\colim_{\catname{B}} \ffunctor{D}$ and
  $\colim_{\catname{A}}{\ffunctor{D} \circ \ffunctor{F}}$ each exists if and only if the other one
  exists, and the canonical comparison morphism
  \[
    \colim_{\catname{A}}{\ffunctor{D} \circ \ffunctor{F}} \to \colim_{\catname{B}}{\ffunctor{D}}
  \]
  is an isomorphism.

  A $2$-final $1$-functor $\ffunctor{F} \colon \catname{A} \to \catname{B}$ between
  $1$-categories (seen as $2$-categories with only the identities as
  $2$-morphisms) is $1$-final, since any diagram is also a $2$-diagram. The
  converse is not true, though: there are $1$-final functors which are not
  $2$-final.
\end{remark}

\begin{theorem}\label{thmFinal}
  Let $\catname{A}$, $\catname{B}$ be two $(2,1)$-categories. A 2-functor
  $\ffunctor{F} \colon \catname{A} \to \catname{B}$ is 2-final
  (\cref{def2Final}) if and only if, for any object $b \in \catname{B}$, the
  slice (2,1)-category $b / \ffunctor{F}$ is nonempty, connected and simply
  connected (\cref{defConn}).
\end{theorem}

We will first prove the backward implication.

Fix a 2-functor $\ffunctor{D} \colon \catname{B} \to \catname{E}$. We will
construct a pseudoinverse to the canonical comparison morphism
\[
  \ccolim_{\catname{A}}{\ffunctor{D} \circ \ffunctor{F}} \to \ccolim_{\catname{B}}{\ffunctor{D}}
\]
This morphism correspond to a family of functors, pseudonatural in $e$:
\[
  \functor{K} \colon [\catname{B}, \catname{E}](\ffunctor{D}, \Delta{e})
  \to [\catname{A}, \catname{E}](\ffunctor{D} \circ \ffunctor{F}, \Delta{e}).
\]
We will construct a pseudoinverse $\ffunctor{L}$ to $\ffunctor{K}$. Given a cone
$\phi \colon \ffunctor{D} \circ \ffunctor{F} \Rightarrow \Delta{e}$, we obtain a
cone $\ffunctor{L}(\phi) \colon \ffunctor{D} \Rightarrow \Delta{e}$ as follows:
\begin{itemize}
\item objects in the slice categories $b / \ffunctor{F}$ define the 1-morphism components
  $\functor{L}(\phi)_b$ (see \cref{defComp}),
\item paths in $b / \ffunctor{F}$ define natural transformations between the
  components (see \cref{defJ}),
\item homotopies between paths ensure the cohesion of these constructions (see \cref{propJInv}).
\end{itemize}

Consider an
arbitrary cone under $\ffunctor{D} \circ \ffunctor{F}$ with vertex $e \in \catname{E}$, that is, a
pseudo natural transformation $\phi \colon \ffunctor{D} \circ \ffunctor{F} \to
\Delta(e)$. We first want to define a cone $\psi$ under $\ffunctor{D}$ with vertex $e$,
using the cone $\phi$.

As a first step, we fix an object $b$ and we want to define the component at $b$
$\psi_b \colon \ffunctor{D}(b) \to e$. Since the slice (2,1)-category $b /
\ffunctor{F}$ is nonempty, we consider the following candidate.
\begin{definition}\label{defComp}
  We fix an object in $b / \ffunctor{F}$, that is an object $a(b) \in
  \catname{A}$ and a morphism $\alpha(b) \colon b \to \ffunctor{F}(a(b))$. Define
  \[
    \psi_{(a(b), \alpha(b))} \colon \begin{tikzcd}
      \ffunctor{D}(b) \arrow[r, "\ffunctor{D}(\alpha(b))"] &
      \ffunctor{D}(\ffunctor{F}(a(b))) \arrow[r, "\phi_{a(b)}"] & e
    \end{tikzcd}
  \]
\end{definition}

We then consider the dependence of $\psi_{(a(b), \alpha(b))}$ on $(a(b),
\alpha(b))$. Fix another object $(a'(b),\alpha'(b)) \in b / \ffunctor{F}$.
Since $b / \ffunctor{F}$ is connected there is a path
\[
  p \colon (a_0, \alpha_0) = (a(b), \alpha(b)) \leadsto (a_n, \alpha_n) =(a'(b),
  \alpha'(b))
\] which can be pictured as:
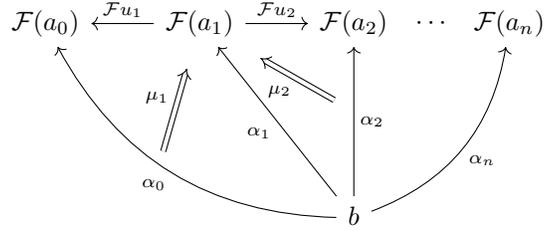
\[
  \begin{tikzcd}[row sep=2cm]
    \ffunctor{F}(a_0) & \ffunctor{F}(a_1) \arrow[l, "\ffunctor{F}u_1", swap] \arrow[r, "\ffunctor{F}u_2"] &
    \ffunctor{F}(a_2) \arrow[r, phantom, "\cdots"] & \ffunctor{F}(a_n) \\
    & & b
    \arrow[llu, "\alpha_0"{name=n0}, bend left]
    \arrow[lu, "\alpha_1"{name=n1}]
    \arrow[u, "\alpha_2"{name=n2}, swap]
    \arrow[ru, "\alpha_n", bend right, swap]
    \arrow[to=1-2, from=n0, "\mu_1", Rightarrow, shorten=0.3cm]
    \arrow[to=1-2, from=n2, "\mu_2", Rightarrow, shorten=0.3cm]
  \end{tikzcd}
\]

Applying the 2-functor $\ffunctor{D}$ and using the cone $\phi$, we obtain the
pasting diagram:
\begin{equation}\label{DiagDefJ}
  \begin{tikzcd}[row sep=1.5cm, column sep=3cm]
    & \ffunctor{D}\ffunctor{F}(a_0) \arrow[rd, ""{name=nphia0}, bend left] & \\
    \ffunctor{D}(b)
    \arrow[ru, ""{name=nalp0}, bend left, swap]
    \arrow[r] \arrow[rd, ""{name=nalp2},  bend right]
    \arrow[rdd, bend right] &
    \ffunctor{D}\ffunctor{F}(a_1)
    \arrow[u, "\ffunctor{D}\ffunctor{F}(u_1)"description]
    \arrow[d, "\ffunctor{D}\ffunctor{F}(u_2)"description]
    \arrow[r, ""{name=nphia1}] & e \\
    & \ffunctor{D}\ffunctor{F}(a_2)
    \arrow[d, phantom, "\vdots", near start]
    \arrow[ru, ""{name=nphia2}, bend right, swap] & \\
    & \ffunctor{D}\ffunctor{F}(a_n) \arrow[ruu, bend right] &
    \arrow[from=nalp0, to=2-2, "\ffunctor{D}(\mu_1)", Rightarrow, shorten=0.2cm, swap]
    \arrow[from=1-2, to=nphia1, "\phi_{u_1}", Rightarrow, shorten=0.6cm, swap]
    \arrow[from=2-2, to=nalp2, "\ffunctor{D}(\mu_2)^{-1}", Rightarrow, shorten=0.2cm, swap]
    \arrow[from=nphia1, to=3-2, "\phi_{u_2}^{-1}", Rightarrow, shorten=0.6cm, swap]
  \end{tikzcd}
\end{equation}
We can thus define:
\begin{definition}\label{defJ}
  Any path $p \colon (a, \alpha) \leadsto (a', \alpha')$ in $b / \ffunctor{F}$ defines a
  2-isomorphism in~$\catname{E}$
  \[
    j(p) \colon \psi_{(a, \alpha)} \to \psi_{(a', \alpha')}
  \]
  as given by the above pasting diagram \ref{DiagDefJ}.
\end{definition}

\begin{proposition}\label{propJInv}
  For any paths $p, p' \colon (a, \alpha) \leadsto (a', \alpha')$ in $b /
  \ffunctor{F}$ with same source and target,
  \[
    j(p) = j(p').
  \]
\end{proposition}
\begin{proof}
  We first prove that two elementary homotopic paths $p \sim_{\text{elem}} p'$ induce the same
  2-isomorphism $j(p) = j(p')$. The four first cases are immediate consequences
  of the pseudonaturality of $\phi$. We can thus assume that
  \[
    p =
    \begin{tikzcd}
      \ffunctor{F}(a_0) & \ffunctor{F}(a_1) \arrow[l, "\ffunctor{F}u", swap]
      \arrow[r, "\ffunctor{F}v"] & \ffunctor{F}(a_2) \\
      & b \arrow[lu, bend left, ""{name=n1}] \arrow[u] \arrow[ru, bend right,
      ""{name=n2}, swap] &
      \arrow[from=n1, to=1-2, "\mu", Rightarrow, shorten=0.2cm]
      \arrow[from=n2, to=1-2, "\nu", Rightarrow, shorten=0.2cm, swap]
    \end{tikzcd}
  \]
  \[
    p' =
    \begin{tikzcd}
      \ffunctor{F}(a_0) \arrow[r, "\ffunctor{F}u'"] & \ffunctor{F}(a_1') &
      \ffunctor{F}(a_2) \arrow[l, "\ffunctor{F}v'", swap] \\
      & b \arrow[lu, bend left] \arrow[u] \arrow[ru, bend right] &
      \arrow[from=2-2, to=1-1, "\mu'", Rightarrow, shorten=0.2cm, swap]
      \arrow[from=2-2, to=1-3, "\nu'", Rightarrow, shorten=0.2cm]
    \end{tikzcd}
  \]
  and that there is a 2-isomorphism $\zeta \colon u'u \Rightarrow v'v$ such that
  \[
    \begin{tikzcd}[column sep=0.25cm, baseline=(current bounding box.center)]
      & {} \arrow[d, Rightarrow, "\mu'", shorten=0.2cm, swap] & \ffunctor{F}(a_1') & \\
      b \arrow[rru, bend left] \arrow[r] \arrow[rrd, bend right] &
      \ffunctor{F}(a_0)
        \arrow[ru, "\ffunctor{F}u'"]
        \arrow[d, Rightarrow,"\mu", shorten=0.2cm, swap]
        \arrow[rr, Rightarrow, "\ffunctor{F}\zeta", shorten=0.5cm]
      & & \ffunctor{F}(a_0) \arrow[lu, "\ffunctor{F}v'", swap] \\
      & {} & \ffunctor{F}(a_1) \arrow[lu, "\ffunctor{F}u"] \arrow[ru,
      "\ffunctor{F}v", swap] &
    \end{tikzcd} \quad = \quad
    \begin{tikzcd}[baseline=(current bounding box.center)]
      & & \ffunctor{F}(a_1') \arrow[d, Rightarrow, "\nu'", shorten=0.2cm] & \\
      b \arrow[rru, bend left] \arrow[rrr] \arrow[rrd, bend right] & & {}
      \arrow[d, Rightarrow, "\nu", shorten=0.2cm]& \ffunctor{F}(a_0)
      \arrow[lu, "\ffunctor{F}v'", swap] \\
      & & \ffunctor{F}(a_1) \arrow[ru, "\ffunctor{F}v", swap] &
    \end{tikzcd}
  \]
  We can apply the functor $\ffunctor{D}$ and express this relation using
  string diagrams (see \autocite[§3.7]{johnsonDimensionalCategories2021}):
  \begin{equation} \label{ZetaMorphSlice}
    \begin{tikzpicture}[baseline=(current bounding box.center)]
      \node[tip] (sta0) at (2, 0) {$\ffunctor{D}\alpha_1'$};
      \node[tip] (end0) at (0, -4) {$\ffunctor{D}\alpha_1$};
      \node[tip] (end1) at (1, -4) {$\ffunctor{D}\ffunctor{F}v$};
      \node[tip] (end2) at (2, -4) {$\ffunctor{D}\ffunctor{F}v'$};
      \node[naturaltr] (dmu0) at (1.5, -1) {$\ffunctor{D}\mu'$};
      \node[naturaltr] (dmu1) at (0.5, -2) {$\ffunctor{D}\mu$};
      \node[naturaltr] (zeta) at (1.5, -3) {$\ffunctor{D}\ffunctor{F}\zeta$};
      \draw (sta0) to[out=270, in=90] (dmu0.north east);
      \draw (dmu0.south west) to[out=270, in=90] (dmu1.north east);
      \draw (dmu0.south east) to[out=270, in=90] (zeta.north east);
      \draw (dmu1.south west) to[out=270, in=90] (end0);
      \draw (dmu1.south east) to[out=270, in=90] (zeta.north west);
      \draw (zeta.south west) to[out=270, in=90] (end1);
      \draw (zeta.south east) to[out=270, in=90] (end2);
    \end{tikzpicture} \xspacedequal{}
    \begin{tikzpicture}[baseline=(current bounding box.center)]
      \node[tip] (sta0) at (2, 0) {$\ffunctor{D}\alpha_1'$};
      \node[tip] (end0) at (0, -4) {$\ffunctor{D}\alpha_1$};
      \node[tip] (end1) at (1, -4) {$\ffunctor{D}\ffunctor{F}v$};
      \node[tip] (end2) at (2, -4) {$\ffunctor{D}\ffunctor{F}v'$};
      \node[naturaltr] (dnu0) at (1.5, -1) {$\ffunctor{D}\nu'$};
      \node[naturaltr] (dnu1) at (0.5, -2) {$\ffunctor{D}\nu$};
      \draw (sta0) to[out=270, in=90] (dnu0.north east);
      \draw (dnu0.south west) to[out=270, in=90] (dnu1.north east);
      \draw (dnu0.south east) to[out=270, in=90] (end2);
      \draw (dnu1.south west) to[out=270, in=90] (end0);
      \draw (dnu1.south east) to[out=270, in=90] (end1);
    \end{tikzpicture}
  \end{equation}
  Similarly, the pseudonaturality of $\phi$ gives the relation:
  \begin{equation} \label{PsNatPhiZeta}
    \begin{tikzpicture}[baseline=(current bounding box.center)]
      \node[tip] (sta0) at (0, 0) {$\ffunctor{D}\ffunctor{F}u$};
      \node[tip] (sta1) at (1, 0) {$\ffunctor{D}\ffunctor{F}u'$};
      \node[tip] (sta2) at (2, 0) {$\phi_{a_1'}$};
      \node[tip] (end0) at (0, -4) {$\phi_{a_1}$};
      \node[naturaltr] (u') at (1.5, -1) {$\phi_{u'}$};
      \node[naturaltr] (u) at (0.5, -2) {$\phi_{u}$};
      \draw (sta0) to[out=270, in=90] (u.north west);
      \draw (sta1) to[out=270, in=90] (u'.north west);
      \draw (sta2) to[out=270, in=90] (u'.north east);
      \draw (u'.south west) to[out=270, in=90] (u.north east);
      \draw (u.south west) to[out=270, in=90] (end0);
    \end{tikzpicture} =
    \begin{tikzpicture}[baseline=(current bounding box.center)]
      \node[tip] (sta0) at (0, 0) {$\ffunctor{D}\ffunctor{F}u$};
      \node[tip] (sta1) at (1, 0) {$\ffunctor{D}\ffunctor{F}u'$};
      \node[tip] (sta2) at (2, 0) {$\phi_{a_1'}$};
      \node[tip] (end0) at (0, -4) {$\phi_{a_1}$};
      \node[naturaltr] (zeta) at (0.5, -1) {$\ffunctor{D}\ffunctor{F}\zeta$};
      \node[naturaltr] (v') at (1.5, -2) {$\phi_{v'}$};
      \node[naturaltr] (v) at (0.5, -3) {$\phi_{v}$};
      \draw (sta0) to[out=270, in=90] (zeta.north west);
      \draw (sta1) to[out=270, in=90] (zeta.north east);
      \draw (sta2) to[out=270, in=90] (v'.north east);
      \draw (zeta.south west) to[out=270, in=90] (v.north west);
      \draw (zeta.south east) to[out=270, in=90] (v'.north west);
      \draw (v'.south west) to[out=270, in=90] (v.north east);
      \draw (v.south west) to[out=270, in=90] (end0);
    \end{tikzpicture}
  \end{equation}
  We can now compute $j(p)$:
  \begin{align*}
    j(p) &=
           \begin{tikzpicture}[baseline=(current bounding box.center)]
             \node[tip] (sta0) at (1, 0) {$\ffunctor{D}\alpha_0$};
             \node[tip] (sta1) at (2, 0) {$\phi_{a_0}$};
             \node[tip] (end0) at (1, -5) {$\ffunctor{D}\alpha_2$};
             \node[tip] (end1) at (2, -5) {$\phi_{a_2}$};
             \node[naturaltr] (dmu) at (0.5, -1) {$\ffunctor{D}\mu$};
             \node[naturaltr] (u) at (1.5, -2) {$\phi_{u}$};
             \node[naturaltr] (v) at (1.5, -3) {$\phi_{v}^{-1}$};
             \node[naturaltr] (dnu) at (0.5, -4) {$\ffunctor{D}\nu^{-1}$};
             \draw (sta0) to[out=270, in=90] (dmu.north east);
             \draw (sta1) to[out=270, in=90] (u.north east);
             \draw (dmu.south west) to[out=270, in=90] (dnu.north west);
             \draw (dmu.south east) to[out=270, in=90] (u.north west);
             \draw (u.south west) to[out=270, in=90] (v.north west);
             \draw (u.south east) to[out=270, in=90] (v.north east);
             \draw (v.south west) to[out=270, in=90] (dnu.north east);
             \draw (v.south east) to[out=270, in=90] (end1);
             \draw (dnu.south east) to[out=270, in=90] (end0);
           \end{tikzpicture} \xspacedequal{\eqref{PsNatPhiZeta}}
           \begin{tikzpicture}[baseline=(current bounding box.center)]
             \node[tip] (sta0) at (1, 0) {};
             \node[tip] (sta1) at (2, 0) {};
             \node[tip] (end0) at (1, -6) {};
             \node[tip] (end1) at (2, -6) {};
             \node[naturaltr] (dmu) at (0.5, -1) {$\ffunctor{D}\mu$};
             \node[naturaltr] (u) at (2.5, -2) {$\phi_{u'}^{-1}$};
             \node[naturaltr] (zeta) at (1.5, -3) {$\ffunctor{D}\ffunctor{F}\zeta$};
             \node[naturaltr] (v) at (2.5, -4) {$\phi_{v'}$};
             \node[naturaltr] (dnu) at (0.5, -5) {$\ffunctor{D}\nu^{-1}$};
             \draw (sta0) to[out=270, in=90] (dmu.north east);
             \draw (sta1) to[out=270, in=90] (u.north west);
             \draw (dmu.south west) to[out=270, in=90] (dnu.north west);
             \draw (dmu.south east) to[out=270, in=90] (zeta.north west);
             \draw (u.south west) to[out=270, in=90] (zeta.north east);
             \draw (u.south east) to[out=270, in=90] (v.north east);
             \draw (zeta.south west) to[out=270, in=90] (dnu.north east);
             \draw (zeta.south east) to[out=270, in=90] (v.north west);
             \draw (v.south west) to[out=270, in=90] (end1);
             \draw (dnu.south east) to[out=270, in=90] (end0);
           \end{tikzpicture} \\
         &\spacedequal{}
           \begin{tikzpicture}[baseline=(current bounding box.center)]
             \node[tip] (sta0) at (1, 0) {};
             \node[tip] (sta1) at (2, 0) {};
             \node[tip] (end0) at (1, -6) {};
             \node[tip] (end1) at (2, -6) {};
             \node[naturaltr] (u) at (2.5, -1) {$\phi_{u'}^{-1}$};
             \node[naturaltr] (dmu) at (0.5, -2) {$\ffunctor{D}\mu$};
             \node[naturaltr] (zeta) at (1.5, -3) {$\ffunctor{D}\ffunctor{F}\zeta$};
             \node[naturaltr] (dnu) at (0.5, -4) {$\ffunctor{D}\nu^{-1}$};
             \node[naturaltr] (v) at (2.5, -5) {$\phi_{v'}$};
             \draw (sta0) to[out=270, in=90] (dmu.north east);
             \draw (sta1) to[out=270, in=90] (u.north west);
             \draw (u.south west) to[out=270, in=90] (zeta.north east);
             \draw (u.south east) to[out=270, in=90] (v.north east);
             \draw (dmu.south west) to[out=270, in=90] (dnu.north west);
             \draw (dmu.south east) to[out=270, in=90] (zeta.north west);
             \draw (zeta.south west) to[out=270, in=90] (dnu.north east);
             \draw (zeta.south east) to[out=270, in=90] (v.north west);
             \draw (dnu.south east) to[out=270, in=90] (end0);
             \draw (v.south west) to[out=270, in=90] (end1);
           \end{tikzpicture} \xspacedequal{\eqref{ZetaMorphSlice}}
           \begin{tikzpicture}[baseline=(current bounding box.center)]
             \node[tip] (sta0) at (1, 0) {$\ffunctor{D}\alpha_0$};
             \node[tip] (sta1) at (2, 0) {$\phi_{a_0}$};
             \node[tip] (end0) at (1, -5) {$\ffunctor{D}\alpha_2$};
             \node[tip] (end1) at (2, -5) {$\phi_{a_2}$};
             \node[naturaltr] (u) at (2.5, -1) {$\phi_{u'}^{-1}$};
             \node[naturaltr] (dmu) at (1.5, -2) {$\ffunctor{D}\mu'^{-1}$};
             \node[naturaltr] (dnu) at (1.5, -3) {$\ffunctor{D}\nu'$};
             \node[naturaltr] (v) at (2.5, -4) {$\phi_{v'}$};
             \draw (sta0) to[out=270, in=90] (dmu.north west);
             \draw (sta1) to[out=270, in=90] (u.north west);
             \draw (u.south west) to[out=270, in=90] (dmu.north east);
             \draw (u.south east) to[out=270, in=90] (v.north east);
             \draw (dmu.south east) to[out=270, in=90] (dnu.north east);
             \draw (dnu.south west) to[out=270, in=90] (end0);
             \draw (dnu.south east) to[out=270, in=90] (v.north west);
             \draw (v.south west) to[out=270, in=90] (end1);
           \end{tikzpicture} \spacedequal{} j(p')
  \end{align*}

  We now show that, for two homotopic paths $p \sim p'$, we have $j(p) = j(p')$. It suffices to show
  that the relation $\mathcal{R}$ on paths defined by
  \[
    p \mathcal{R} p' \Longleftrightarrow j(p) = j(p')
  \]
  is reflexive, symmetric, transitive and congruent, since we have already
  proved that it contains $\sim_{\text{elem}}$. The three first
  properties are obviously satisfied. The last one is a direct consequence of
  the compatibility of $j$ with the concatenation of paths: $j(p \cdot p') =
  j(p')j(p)$.

  Since $b / \ffunctor{F}$ is simply connected by hypothesis and $j$ is homotopy invariant, the $2$-isomorphism $j(p)$
  only depends on the source and the target of $p$. Hence for any two objects
  $(a, \alpha)$ and $(a', \alpha')$ in $b / \ffunctor{F}$, there is a unique
  $2$-isomorphism $\psi_{(a,
    \alpha)} \Rightarrow \psi_{(a', \alpha')}$ in $\catname{E}$ induced by a path in $b/
  \ffunctor{F}$.
\end{proof}

\begin{definition}
  Given a morphism $u \colon b \to b'$ in \catname{B}, there is a
  base change functor:
  \[
    u^* \colon \left\{
      \begin{array}{lll}
        b' / \ffunctor{F} & \to & b / \ffunctor{F} \\
        (x, \chi) & \mapsto & (x, \chi \circ u) \\
        (v, \nu)  & \mapsto & (v, \nu \cdot u)
      \end{array}
    \right.
  \]
  Note that this functor also extends to a function between the respective sets of
  paths.
\end{definition}

\begin{proposition}
  The application $j$ maps base change to whiskering:
  \[
    j(u^*p) = j(p) \cdot \ffunctor{D}u
  \]
\end{proposition}

We can now use the above properties to construct a cone $\psi$ under $\ffunctor{D}$ with
vertex~$e$. For any $b \in \catname{B}$, fix an arbitrary object $(a(b),
\alpha(b))$ in $b / \ffunctor{F}$. This defines the components $\psi_b =
\psi_{(a(b), \alpha(b))}$, as stated in \cref{defComp}. For a morphism $u \colon b \to b'$, note that
$\psi_{(a(b'), \alpha(b')} \circ \ffunctor{D}u =
\psi_{u^*(a(b'),\alpha(b'))}$; hence we can define $\psi_u$ as the unique
2-isomorphism $j(p)$ induced by any path $p \colon u^*(a(b'), \alpha(b')) \leadsto (a(b),
\alpha(b))$. We must check that $\psi$ is indeed a pseudonatural
transformation. The compatibility of $j$ with the whiskering and the
concatenation of paths implies the required compatibility of $\psi$ with the
composition of morphisms.
It remains to check the compatibility with 2-morphisms. Let $u, u' \colon b
\to b'$ be two parallel 1-morphisms and $\delta \colon u \Rightarrow u'$ be a
2-morphism in \catname{B}. By unicity of
2-morphisms induced by a path (\cref{propJInv}), it suffices to check that the pasting
\begin{equation}\label{diagComp2Morph}
  \begin{tikzcd}
    \ffunctor{D}b' \arrow[rrd, bend left] & & \\
    & & e \\
    \ffunctor{D}b \arrow[uu, bend left=2cm, "\ffunctor{D}u"{name=n1, description}] \arrow[uu,
    bend right=2cm, "\ffunctor{D}u'"{name=n2, description}] \arrow[from=n1, to=n2, Rightarrow,
    shorten=0.1cm, "\ffunctor{D}\delta"] \arrow[from=n2, to=2-3, Rightarrow,
    shorten=0.1cm, "\psi_{u'}"] \arrow[rru, bend right]
  \end{tikzcd}
\end{equation}
is induced by a path. Indeed, fix a path $p \colon (u')^*(a(b'), \alpha(b'))
\leadsto (a(b), \alpha(b))$ and recall that, by definition, $\psi_{u'} =
j(p)$.
We consider the path $p'$ of length one:
\[
  p' \spacedequal
  (a(b'), u'\alpha(b')) \xleftarrow{(\id, \alpha(b')\delta)} (a(b'),
  u'\alpha(b')) \spacedequal
  \begin{tikzcd}[baseline=(current bounding box.center), sep=0.25cm]
    a(b') \arrow[rr, equal] & & a(b') \\
    b' \arrow[u] \arrow[rr, equal] & & b' \arrow[u] \\
    & b \arrow[lu, bend left, "u"{name=n1}, pos=0.9] \arrow[ru, bend right,
    "u'"{name=n2}, swap, pos=0.9]
    \arrow[from=n1, to=n2, Rightarrow, "\delta", shorten=0.4cm]
  \end{tikzcd}
\]
The above pasting \eqref{diagComp2Morph} is then induced by the concatenation $p' \cdot p$ of $p'$ and $p$.

Through similar arguments, we can see that any other choice of the objects
$(a(b),\alpha(b))_{b \in \catname{B}}$ leads to an isomorphic cone.

From now on, we assume that the objects $(a(b), \alpha(b))_{b \in \catname{B}}$ are fixed and
we write $\functor{L}(\phi)$ for the cone $\psi$ under $\ffunctor{D}$ induced
from the cone $\phi$ under $\ffunctor{D} \circ \ffunctor{F}$. Since we will not
work with a single fixed cone $\phi$ anymore, we should write $j_{\phi}$ instead
of $j$.

We would like to extend this mapping $\phi \mapsto \functor{L}{\phi}$ to a
functor
\[
  \functor{L} \colon [\catname{A},\catname{E}](\ffunctor{D} \circ \ffunctor{F}, \Delta{e}) \to
  [\catname{B}, \catname{E}](\ffunctor{D}, \Delta{e}).
\]
We use the proposition:
\begin{proposition}\label{propModifJ}
  Let $m \colon \phi \to \phi'$ be a modification between two cones
  \[
    \phi, \phi' \colon \ffunctor{D} \circ \ffunctor{F} \Rightarrow \Delta{e}
  \]
  For any path $p \colon (a, \alpha) \leadsto (a', \alpha')$ in $b /
  \ffunctor{F}$, we have the following equality:
  \begin{equation}\label{eqModifJ}
    \begin{tikzcd}[column sep=0.35cm, row sep=1cm]
      e \arrow[rr, equal] \arrow[from=d, bend left, "\phi_a", pos=0.45] & &
      e \arrow[from=d, bend left, "\phi_{a'}"{name=n1}, pos=0.45]
      \arrow[from=d, bend right, "\phi'_{a'}"{name=n2}, pos=0.45, swap]
      \arrow[from=n1, to=n2, "m_{a'}", Rightarrow, shorten=0.2cm] \\
      \ffunctor{D}\ffunctor{F}(a) \arrow[from=rd] \arrow[rr, Rightarrow,
      "j_{\phi}(p)", shorten=0.5cm] & &
      \ffunctor{D}\ffunctor{F}(a') \arrow[from=ld] \\
      & \ffunctor{D}(b)
    \end{tikzcd} \spacedequal
    \begin{tikzcd}[column sep=0.35cm, row sep=1cm]
      e \arrow[rr, equal] \arrow[from=d, bend left, "\phi_a"{name=n1}, pos=0.45]
      \arrow[from=d, bend right, "\phi'_a"{name=n2}, pos=0.45, swap]
      \arrow[from=n1, to=n2, "m_a", Rightarrow, shorten=0.2cm] & &
      e \arrow[from=d, bend right, "\phi_{a'}", pos=0.45]\\
      \ffunctor{D}\ffunctor{F}(a) \arrow[from=rd] \arrow[rr, Rightarrow,
      "j_{\phi'}(p)", shorten=0.5cm] & &
      \ffunctor{D}\ffunctor{F}(a') \arrow[from=ld] \\
      & \ffunctor{D}(b)
    \end{tikzcd}
  \end{equation}
\end{proposition}
\begin{proof}
  For an empty path $p$, the proposition reduce to the tautology $m_a = m_a$.
  For a path $p = (a, \alpha) \xrightarrow{(u, \mu)} (a',  \alpha')$ of length 1,
  we can decompose the equation as:
  \[
    \begin{tikzcd}[column sep=0.35cm, row sep=1cm]
      e \arrow[rr, equal] \arrow[from=d, bend left, "\phi_a"{name=n0}, pos=0.45,
      swap] & &
      e \arrow[from=d, bend left, "\phi_{a'}"{name=n1}, pos=0.45]
      \arrow[from=d, bend right, "\phi'_{a'}"{name=n2}, pos=0.45, swap]
      \arrow[from=n0, to=n1, "\phi_{u}^{-1}", Rightarrow, shorten=0.2cm]
      \arrow[from=n1, to=n2, "m_{a'}", Rightarrow, shorten=0.2cm] \\
      \ffunctor{D}\ffunctor{F}(a) \arrow[from=rd, ""{name=n3}]
      \arrow[rr, "\ffunctor{D}\ffunctor{F}u"] & &
      \ffunctor{D}\ffunctor{F}(a') \arrow[from=ld, ""{name=n4}, swap]
      \arrow[from=n3, to=n4, Rightarrow, "\ffunctor{D}\mu^{-1}", shorten=0.2cm] \\
      & \ffunctor{D}(b)
    \end{tikzcd} \spacedequal
    \begin{tikzcd}[column sep=0.35cm, row sep=1cm]
      e \arrow[rr, equal] \arrow[from=d, bend left, "\phi_a"{name=n1}, pos=0.45]
      \arrow[from=d, bend right, "\phi'_a"{name=n2}, pos=0.45, swap]
      \arrow[from=n1, to=n2, "m_a", Rightarrow, shorten=0.2cm] & &
      e \arrow[from=d, bend right, "\phi_{a'}"{name=n0}, pos=0.45]
      \arrow[from=n2, to=n0, "\phi_u^{'-1}", Rightarrow, shorten=0.3cm] \\
      \ffunctor{D}\ffunctor{F}(a) \arrow[from=rd, ""{name=n3}]
      \arrow[rr, "\ffunctor{D}\ffunctor{F}u"] & &
      \ffunctor{D}\ffunctor{F}(a') \arrow[from=ld, ""{name=n4}, swap]
      \arrow[from=n3, to=n4, Rightarrow, "\ffunctor{D}\mu^{-1}", shorten=0.2cm]\\
      & \ffunctor{D}(b)
    \end{tikzcd}
  \]
  The lower parts of these diagrams are the same and the upper parts are equal,
  by the property of the modification $m$. Hence \cref{eqModifJ} holds for a
  path $p = ((a, \alpha) \xrightarrow{(u, \mu)} (a',  \alpha'))$. A similar
  decomposition of the diagrams shows that it also holds for a path $p = ((a,
  \alpha) \xleftarrow{(u, \mu)} (a',  \alpha'))$ of length one in the reverse direction.

  Since $j_\phi$ and $j_{\phi'}$ are compatible with paths concatenation, if
  \cref{eqModifJ} holds for two composable paths $p$ and $p'$, it also holds for
  their concatenation $pp'$. We can thus conclude that it holds for any path $p$, as
  the path $p$ is generated by paths of length 1.
\end{proof}
This property directly implies that the components
\[
  \functor{L}(m)_b \: = \: \begin{tikzcd}[column sep=1.2cm]
   \ffunctor{D}(b) \arrow[r] & \ffunctor{D}(a(b))
    \arrow[r, bend left, ""{name=n1}, pos=0.4, swap] \arrow[r, bend right,
    ""{name=n2}, pos=0.4]
    & e
    \arrow[from=n1, to=n2, "m_{a(b)}", Rightarrow]
  \end{tikzcd}
\]
define a 2-morphism $\functor{L}(\phi) \to \functor{L}(\phi')$. The
functoriality of $\functor{L}$ is straightforward.

Now consider the canonical functor
\[
  \functor{K} \colon \left\{
    \begin{array}{lll}
      [\catname{B}, \catname{E}](\ffunctor{D}, \Delta{e}) & \to
      & [\catname{A}, \catname{E}](\ffunctor{D} \circ \ffunctor{F}, \Delta{e}) \\
      \psi_{\bullet} & \mapsto & \psi_{\ffunctor{F}(\bullet)} \\
      m_{\bullet} & \mapsto & m_{\ffunctor{F}(\bullet)}
    \end{array}
  \right.
\]
sending cones under \ffunctor{D} with vertex $e$ to cones under $\ffunctor{D}
\circ \ffunctor{F}$ with vertex $e$. We are now ready to show that
$\functor{L}$ and $\functor{K}$ are mutual pseudo-inverses.

\begin{proposition}\label{propKLId}
  There is a natural isomorphism $\eta \colon \id \Rightarrow \functor{K}\functor{L}$.
\end{proposition}
\begin{proof}
  Let $\phi \in  [\catname{A}, \catname{E}](\ffunctor{D} \circ \ffunctor{F},
  \Delta{e})$ be a cone under $\ffunctor{D} \circ \ffunctor{F}$ with vertex $e$.
  We will write $j = j_{\phi}$.

  We want to define the component $\eta_{\phi} \colon \phi \to
  \functor{K}\functor{L}\phi$ of $\eta$ at $\phi$. Since $\eta_{\phi}$ must be
  a modification, we have to define its components at each object $a_0 \in \catname{A}$:
  \[
    \eta_{\phi, a_0} \colon \phi_{a_0} \Rightarrow \phi_{a(\ffunctor{F}a_0)} \circ \ffunctor{D}\alpha{\ffunctor{F}a_0}.
  \]
  Both $(a_0, \id_{\ffunctor{F}a_0})$ and $(a(\ffunctor{F}a_0),
  \alpha(\ffunctor{F}a_0))$ are objects of $\ffunctor{F}a_0 / \ffunctor{F}$, which
  is connected. Hence, there is a path in $\ffunctor{F}a_0 / \ffunctor{F}$:
  \[
    p \colon (a_0, \id_{\ffunctor{F}a_0}) \leadsto (a(\ffunctor{F}a_0),\alpha(\ffunctor{F}a_0))
  \]

  We have to check that $\eta_{\phi}$ is a modification. That is, for any
  morphism $f \colon a_0 \to a_1$ in $\catname{A}$, we have to check the
  commutativity of:
  \[
    \begin{tikzcd}[column sep=2cm]
      \phi_a
        \arrow[r, "\eta_{\phi, a_0}", Rightarrow]
        \arrow[d, "\phi_f", Rightarrow]
      & \phi_{a(\ffunctor{F}a_0)} \circ \ffunctor{D}\alpha(\ffunctor{F}a_0)
        \arrow[d, "\functor{L}(\phi)_{\ffunctor{F}f}", Rightarrow] \\
      \phi_{a_1} \circ \ffunctor{D}\ffunctor{F}f
        \arrow[r, "\eta_{\phi, a_1} \cdot \ffunctor{D}\ffunctor{F}f", Rightarrow]
      & \phi_{a(\ffunctor{F}a_1)} \circ \ffunctor{D}\alpha(\ffunctor{F}a_1)
      \circ \ffunctor{D}\ffunctor{F}f
    \end{tikzcd}
  \]
  We first remark that there is a path $p_0 = ((a_0, \id) \xrightarrow{(f, \id)}
  (a_1, \ffunctor{F}f))$ in $\ffunctor{F}a_0 / \ffunctor{F}$ and the induced 2-isomorphism is $\phi_f = j(p_0)$.
  Moreover, expanding the definitions, we have
  \begin{align*}
    &\eta_{\phi, a_0} = j(p_1) && \text{for some } p_1 \colon (a_0, \id) \leadsto (a(\ffunctor{F}a_0),\alpha(\ffunctor{F}a_0)) \\
    &\eta_{\phi, a_1} = j(p_2) && \text{for some } p_2 \colon (a_1, \id) \leadsto (a(\ffunctor{F}a_1),\alpha(\ffunctor{F}a_1)) \\
    & \functor{L}(\phi)_{\ffunctor{F}f} = j(p_3) && \text{for some } p_3 \colon (a(\ffunctor{F}a_0), \alpha(\ffunctor{F}a_0)) \leadsto \ffunctor{F}(f)^*(a(\ffunctor{F}a_1), \alpha(\ffunctor{F}a_1))
  \end{align*}
  where $p_1$ and $p_3$ are paths in $\ffunctor{F}a_0 / \ffunctor{F}$, and $p_2$
  is a path in $\ffunctor{F}a_1 / \ffunctor{F}$.
  We can check that $p_1 \cdot p_3$ and $p_0 \cdot \ffunctor{F}(f)^*p_2$
  are paths
  \[
    (a_0, \id) \leadsto \ffunctor{F}(f)^*(a(\ffunctor{F}a_1),\alpha(\ffunctor{F}a_1)).
  \]
  Hence
  \begin{align*}
    (\eta_{\phi, a_1} \cdot \ffunctor{D}\ffunctor{F}f) \circ \phi_f
    &= j(\ffunctor{F}(f)^*p_2) \circ j(p_0) \\
    &= j(p_0 \cdot \ffunctor{F}(f)^*p_2) \\
    &= j(p_1 \cdot p_3) \\
    &= j(p_3) \circ j(p_1) \\
    &= \functor{L}(\phi)_{\ffunctor{F}f} \circ \eta_{\phi, a_0} \\
  \end{align*}

  We also have to check the naturality of $\eta$. For any modification $m \colon
  \phi \to \phi'$, we want the commutativity of the square:
  \[
    \begin{tikzcd}
      \phi \arrow[r, "\eta_{\phi}"] \arrow[d, "m"] &
      \functor{K}\functor{L}\phi \arrow[d, "\functor{K}\functor{L}m"] \\
      \phi' \arrow[r, "\eta_{\phi'}"] &
      \functor{K}\functor{L}\phi'
    \end{tikzcd}
  \]
  That is, for any object $a_0 \in \catname{A}$:
   \[
    \begin{tikzcd}
      \phi_{a_0} \arrow[r, "j_{\phi}(p)"] \arrow[d, "m_{a_0}"] &
      \phi_{a(\ffunctor{F}a_0)} \circ \ffunctor{D}\alpha(\ffunctor{F}a_0) \arrow[d, "m_{\ffunctor{F}a_0}"] \\
      \phi'_{a_0} \arrow[r, "j_{\phi'}(p)"] &
      \functor{K}\functor{L}\phi'_{a_0}
    \end{tikzcd}
  \]
  where $p \colon (a_0, \id) \leadsto (a(\ffunctor{F}a_0), \alpha(\ffunctor{F}a_0))$
  is a path in $\ffunctor{F}a_0 / \ffunctor{F}$. This last square commutes by \cref{propModifJ}.
\end{proof}

In the reverse direction we show:
\begin{proposition}\label{propLKId}
  There is a natural isomorphism $\epsilon : \functor{L}\functor{K} \Rightarrow \id$.
\end{proposition}
\begin{proof}
  Fix a cone $\psi \colon \ffunctor{D} \Rightarrow \Delta{e}$ under $\ffunctor{D}$. Write $\psi' =
  \functor{L}\functor{K}(\psi)$. For any $b \in \catname{B}$, we have:
  \begin{align*}
    \psi'_b &= \functor{K}(\psi)_{a(b)} \circ \ffunctor{D}(\alpha(b))
            = \psi_{\ffunctor{F}(a(b))} \circ \ffunctor{D}(\alpha(b))
  \end{align*}
  Hence we can define a 2-morphism $\epsilon_{\psi,b} \colon \psi'_b \Rightarrow \psi_b$ in $\catname{E}$ by:
  \begin{align*}
    \epsilon_{\psi, b} = \psi_{\alpha(b)}
  \end{align*}
  When $b$ ranges over all objects of $\catname{B}$, these morphisms then form a modification $\epsilon_{\psi} \colon \psi' \to \psi$.
  Indeed for any morphism $(u, \mu) \colon (a, \alpha) \to (a', \alpha')$ in $b
  / \ffunctor{F}$, we have:
  \[
    \begin{tikzcd}[row sep=1.5cm, column sep=1.5cm]
      & e \\
      \ffunctor{D}\ffunctor{F}(a') \arrow[ru, bend left] \arrow[ru, Rightarrow,
      shorten=0.2cm, swap, "\psi_{\ffunctor{F}(u)}"] & \ffunctor{D}\ffunctor{F}(a) \arrow[l] \arrow[u] \\
      & \ffunctor{D}(b) \arrow[lu, bend left, ""{name=n2}, swap] \arrow[u] \arrow[uu, bend right=3cm,
      ""{name=n1}]
      \arrow[from=2-2, to=n1, Rightarrow, "\psi_{\alpha}"]
      \arrow[from=n2, to=2-2, Rightarrow, "\ffunctor{D}\mu"]
    \end{tikzcd} \spacedequal
    \begin{tikzcd}[row sep=1.5cm, column sep=1.5cm]
      & e \\
      \ffunctor{D}\ffunctor{F}(a') \arrow[ru, bend left] & \\
      & \ffunctor{D}(b) \arrow[lu, bend left] \arrow[uu, ""{name=n1}]
      \arrow[from=2-1, to=n1, Rightarrow, "\psi_{\alpha'}"]
    \end{tikzcd}
  \]
  which implies a similar formula for any path $p \colon (a', \alpha') \leadsto (a, \alpha)$ in $b / \ffunctor{F}$:
  \[
    \begin{tikzcd}[row sep=1.5cm, column sep=1.5cm]
      & e \\
      \ffunctor{D}\ffunctor{F}(a') \arrow[ru, bend left] \arrow[r, Rightarrow, "j_{\functor{K}(\psi)}(p)"] & \ffunctor{D}\ffunctor{F}(a) \arrow[u] \\
      & \ffunctor{D}(b) \arrow[lu, bend left] \arrow[u] \arrow[uu, bend right=3cm,
      ""{name=n1}]
      \arrow[from=2-2, to=n1, Rightarrow, "\psi_{\alpha}"]
    \end{tikzcd} \spacedequal
    \begin{tikzcd}[row sep=1.5cm, column sep=1.5cm]
      & e \\
      \ffunctor{D}\ffunctor{F}(a') \arrow[ru, bend left] & \\
      & \ffunctor{D}(b) \arrow[lu, bend left] \arrow[uu, ""{name=n1}]
      \arrow[from=2-1, to=n1, Rightarrow, "\psi_{\alpha'}"]
    \end{tikzcd}
  \]
  This in turn implies that $\epsilon_{\psi}$ is a modification. Fix a morphism
  $u \colon b \to b'$ and consider a path $p \colon u^*(a(b'),\alpha(b'))
  \leadsto (a(b), \alpha(b))$ (hence we have $\psi'_u =
  j_{\functor{K}(\psi)}(p)$). We check the modification axiom at $u$:
  \begin{align*}
    &\begin{tikzcd}[ampersand replacement=\&, row sep=1.2cm]
      e \arrow[r, equal] \& e \\
      \ffunctor{D}\ffunctor{F}(a(b')) \arrow[u] \arrow[r, Rightarrow, "\psi'_u"] \& \ffunctor{D}\ffunctor{F}(a(b)) \arrow[u] \\
      \ffunctor{D}(b') \arrow[u] \& \ffunctor{D}(b) \arrow[l] \arrow[u] \arrow[uu, bend
      right=3cm, ""{name=n1}] \arrow[from=2-2, to=n1, Rightarrow,
      "\epsilon_{\psi, b}"]
    \end{tikzcd} \spacedequal
    \begin{tikzcd}[ampersand replacement=\&, row sep=1.2cm]
      \& e \\
      \ffunctor{D}\ffunctor{F}(a(b')) \arrow[ru, bend left] \arrow[r, Rightarrow, "j_{\functor{K}(\psi)}(p)"] \& \ffunctor{D}\ffunctor{F}(a(b)) \arrow[u] \\
      \& \ffunctor{D}(b) \arrow[lu, bend left] \arrow[u] \arrow[uu, bend
      right=3cm, ""{name=n1}] \arrow[from=2-2, to=n1, Rightarrow,
      "\psi_{\alpha(b)}"]
    \end{tikzcd} \\ &=
                      \begin{tikzcd}[ampersand replacement=\&]
                        \& e \\
                        \ffunctor{D}\ffunctor{F}(a(b')) \arrow[ru, bend left] \& \\
                        \& \ffunctor{D}(b) \arrow[lu, bend left] \arrow[uu, ""{name=n1}]
                        \arrow[from=2-1, to=n1, Rightarrow, "\psi_{\alpha(b')u}"]
                      \end{tikzcd} =
    \begin{tikzcd}[ampersand replacement=\&]
      \& e \arrow[r, equal] \& e\\
      \ffunctor{D}\ffunctor{F}(a(b')) \arrow[ru, bend left] \& \& \\
      \& \ffunctor{D}(b') \arrow[lu, bend left] \arrow[uu, ""{name=n1}] \arrow[ruu, Rightarrow, shorten=0.1cm, "\psi_u"]
      \& \ffunctor{D}(b) \arrow[l] \arrow[uu]
      \arrow[from=2-1, to=n1, Rightarrow, "\epsilon_{\psi, b'}"]
    \end{tikzcd}
  \end{align*}
  Finally we have to check the naturality of $\epsilon \colon \functor{L}\functor{K} \to
  \id$, that is, for any modification of cones $m \colon \psi \to \psi'$, the
  commutativity of the square:
  \[
    \begin{tikzcd}
      \functor{L}\functor{K}\psi \arrow[r, "\epsilon_{\psi}"] \arrow[d,
      "\functor{L}\functor{K}m"] & \psi \arrow[d, "m"] \\
      \functor{L}\functor{K}\psi' \arrow[r, "\epsilon_{\psi'}"] & \psi'
    \end{tikzcd}
  \]
  Indeed, for any object $b \in \catname{B}$:
  \[
    \epsilon_{\psi', b} \circ (\functor{L}\functor{K}m)_b = \psi'_{\alpha(b)}
    \circ m_{a(b)}\alpha(b) = m_{b} \circ \psi_{\alpha(b)}.
  \]
\end{proof}

Putting together \cref{propKLId} and \cref{propLKId} we deduce:
\begin{proposition}
  The canonical functor
  \[
    \functor{K} \colon [\catname{B}, \catname{E}](\ffunctor{D}, \Delta{e}) \to [\catname{A},
    \catname{E}](\ffunctor{D} \circ \ffunctor{F}, \Delta{e})
  \]
  is an equivalence.
\end{proposition}

Since this is true for any object $e$ of $\catname{E}$, clearly
$\ccolim \ffunctor{D}$ exists if and only if $\ccolim \ffunctor{D} \circ
\ffunctor{F}$ exists and, if it is the case, they are canonically equivalent.

We have thus proved one implication of \cref{thmFinal}:
\begin{proposition}
  Let $\ffunctor{F} \colon \catname{A} \to \catname{B}$ be a $(2,1)$-functor.
  If for any object $b \in \catname{B}$, the slice $(2,1)$-category $b /
  \ffunctor{F}$ is nonempty, connected and simply connected, then the $(2,
  1)$-functor $\ffunctor{F}$ is $2$-final.
\end{proposition}

The reverse implication is proved by observing the following fact:
\begin{proposition}\label{propPiColimit}
  Let $\ffunctor{F} \colon \catname{A} \to \catname{B}$ be a $(2, 1)$-functor.
  Then
  \[
    \Pi_1(b / \ffunctor{F}) \cong \ccolim_{a \in \catname{A}}\catname{B}(b , \ffunctor{F}a).
  \]
\end{proposition}
\begin{proof}
 The wanted equivalence can be proved by constructing a family of equivalences,
 pseudonatural in the category $T$:
 \[
   C_{T} \colon [\catname{A}, \catcat](\catname{B}(b, \ffunctor{F}-), \Delta{T})
   \cong [\Pi_1(b / \ffunctor{F}), T].
 \]
 Fix $\psi \colon \catname{B}(b, \ffunctor{F}-) \Rightarrow \Delta{T}$ a
 pseudonatural transformation. We want to define a functor $C_{T}(\psi) \colon
 \Pi_1(b / \ffunctor{F}) \to T$.

 For any object $(a, \alpha \colon b \to \ffunctor{F}(a))$ of $\Pi_1(b /
 \ffunctor{F})$, set:
 \[
   C_{T}(\psi)(a, \alpha) = \psi_a(\alpha)
 \]

 For any morphism $(u, \mu \colon u\alpha \Rightarrow \alpha') \colon (a,
 \alpha) \to (a', \alpha')$ of $b /
 \ffunctor{F}$, define the composite isomorphism:
 \[
   C_{T}(\psi)(u, \mu) :
   \begin{tikzcd}
    \psi_a(\alpha) \arrow[r, "(\psi_u)_{\alpha}"] & \psi_{a'}(u \circ \alpha)
    \arrow[r, "\psi_{a'}(\mu)"] & \psi_{a'}(\alpha')
   \end{tikzcd}
 \]
 This can be extended to paths using the relations:
 \begin{align*}
   C_{T}(\psi)((a', \alpha') \xleftarrow{(u, \mu)} (a, \alpha)) &= C_{T}(\psi)((a, \alpha) \xrightarrow{(u, \mu)} (a', \alpha'))^{-1} \\
   C_{T}(\psi)((a, \alpha)) &= \id_{\psi_a(\alpha)} \\
   C_{T}(\psi)(p \cdot p') &= C_{T}(\psi)(p') \circ C_{T}(\psi)(p)
 \end{align*}
 On can check that such a definition is homotopy invariant, and gives a
 well-defined functor $C_{T}(\psi) \colon \Pi_1(b / \ffunctor{F}) \to T$.

 For a modification $m \colon \psi \to \psi'$, we define a natural transformation
 \[
   C_{T}(m) \colon C_{T}(\psi) \Rightarrow C_{T}(\psi')
 \]
 with components:
 \begin{equation}\label{defCModif}
   C_{T}(m)_{(a, \alpha)} = (m_a)_{\alpha}
 \end{equation}

 To show that $C_T$ is an equivalence, we show that it is a fully faithful and
 essentially surjective functor.

 Indeed it is clear that \eqref{defCModif} defines a bijection
 between modifications $\psi \to \psi'$ and natural transformations $C_{T}(\psi)
 \Rightarrow C_{T}(\psi')$. Hence $C_T$ is fully faithful.

 Moreover, given any functor $F \colon \Pi_1(b / \ffunctor{F}) \to T$, one can
 define a pseudonatural transformation $\psi \colon \catname{B}(b,
 \ffunctor{F}-) \to \Delta{T}$ by:
 \[
  \psi_a : \left\{
    \begin{array}{lll}
      \catname{B}(b, \ffunctor{F}a) & \to & T \\
      \alpha & \mapsto & F(a, \alpha) \\
      \nu & \mapsto & F(\id_a, \nu)
    \end{array}
  \right.
 \]
 \[
  (\psi_u)_{\alpha} \colon
      F(a, \alpha) \xrightarrow{F(u, \id)} F(a', u \circ \alpha)
 \]
 for any object $a$ and morphism $u \colon a \to a'$ of $\catname{A}$. It is
 straightforward to check:
 \[
   F = C_T(\psi)
 \]
\end{proof}

We can now prove:
\begin{proposition}\label{propFinalPiTriv}
  Let $\ffunctor{F} \colon \catname{A} \to \catname{B}$ be a $2$-final $(2,
  1)$-functor. Then, for any object $b$ in $\catname{B}$:
  \[
    \Pi_1(b / \ffunctor{F}) \cong 1
  \]
\end{proposition}
\begin{proof}
 We have a chain of equivalences:
 \[
    \Pi_1(b / \ffunctor{F})
    \overset{\ref{propPiColimit}}{\cong} \ccolim_{a \in
      \catname{A}}\catname{B}(b, \ffunctor{F}a)
    \overset{(1)}{\cong} \ccolim_{b' \in
      \catname{B}}\catname{B}(b, b')
    \overset{(2)}{\cong} 1
  \]
  The equivalence $(1)$ is an application of the $2$-finality of $\ffunctor{F}$.
  The equivalence $(2)$ is a consequence of the Yoneda lemma for $2$-categories.
  Indeed we have the chain of equivalences, for any $1$-category $T$, and
  pseudonatural in $T$:
  \[
    [\catname{B}, \catcat](\catname{B}(b, -), \Delta{T}) \cong \Delta{T}(b)
    \cong T \cong \catcat(1, T)
  \]
\end{proof}

By combining \cref{propFinalPiTriv} and \cref{rmkPiConn}, we have:
\begin{proposition}
  Let $\ffunctor{F} \colon \catname{A} \to \catname{B}$ be a $2$-final $(2,
  1)$-functor. Then, for any object $b$ in $\catname{B}$, the $(2, 1)$-category
  $b / \ffunctor{F}$ is nonempty, connected and simply connected.
\end{proposition}

There is a dual notion of \emph{2-initial 2-functor}, with a dual criterion,
proven by a duality argument.

\begin{definition}
 Let $\ffunctor{F} \colon \catname{A} \to \catname{B}$ be a $2$-functor between
 $(2,1)$-categories. The 2-functor is said to be \emph{2-initial} if, for any
 $2$-diagram $\ffunctor{D} \colon \catname{B} \to \catname{E}$, each of the bilimits
 $\llim_{\catname{B}}{\ffunctor{D}}$ and $\llim_{\catname{A}}{\ffunctor{D} \circ
   \ffunctor{F}}$ exists whenever the other one exists, and the canonical
 comparison $1$-morphism
 \[
   \llim_{\catname{B}}{\ffunctor{D}} \to \llim_{\catname{A}}{\ffunctor{D} \circ \ffunctor{F}}
 \]
 is an equivalence.
\end{definition}

\begin{proposition}
  Let $\ffunctor{F} \colon \catname{A} \to \catname{B}$ be a 2-functor between
  $(2,1)$-categories. The 2-functor $\ffunctor{F}$ is initial if and only if the
  2-functor $\catop{\ffunctor{F}} \colon \catop{\catname{A}} \to
  \catop{\catname{B}}$ is final.
\end{proposition}

\begin{theorem}
  Let $\ffunctor{F} \colon \catname{A} \to \catname{B}$ be a 2-functor between
  $(2,1)$-categories. The 2-functor $\ffunctor{F}$ is initial if and only if,
  for any object $b \in B$, the slice $(2, 1)$-category $\ffunctor{F}/b$ is
  nonempty, connected and simply connected.
\end{theorem}

\section{Further directions}

There are various direction in which one may try to improve the finality
criterion presented in the previous section.

The most straightforward one is to work in the context of bicategories, where
composition of $1$-morphisms is only associative up to isomorphism. One should note
that the correct notions of $2$-finality for pseudofunctors between bicategories
with invertible $2$-morphisms should be weakened to include any pseudofunctor
as diagram, and not only the strict ones as we do in \cref{def2Final}. Since
not all pseudofunctors can be strictified (\autocite{lackBicatNotTriequivalent2007}),
the analogous result for pseudofunctors is not a direct corollary of \cref{thmFinal}.

Another natural route is to prove it for higher dimensions $n$. An analogous
combinatorial proof would require a combinatorial presentation of higher
homotopies in an $n$-category, and probably to set up a machinery for working
inductively on the dimension. An alternative, potentially more reasonable
approach may be to adapt Lurie's topological proof to finite dimensions.

\printbibliography

\end{document}